\documentclass[11pt]{article}

\usepackage{amsmath}
\usepackage{amssymb}
\usepackage{amsthm}

\usepackage{titlesec}

\usepackage{enumitem}

\newtheorem{thm}{Theorem}[section]
\newtheorem{lemma}[thm]{Lemma}
\newtheorem{prop}[thm]{Proposition}
\newtheorem{conj}[thm]{Conjecture}
\newtheorem{cor}[thm]{Corollary}
\theoremstyle{definition}
\newtheorem{defin}[thm]{Definition}
\theoremstyle{remark}

\title{\normalsize{\uppercase{\bf{Nodal intersections of random eigenfunctions against a segment on the 2-dimensional torus}}}}
\author{Riccardo W. Maffucci\footnote{riccardo.maffucci@kcl.ac.uk}\\King's College London\\WC2R 2LS United Kingdom}
\date{}

\begin{document}
\titleformat{\section}
  {\normalfont\scshape\centering\bf}{\thesection}{1em}{}
\titleformat{\subsection}
  {\normalfont\scshape\bf}{\thesubsection}{1em}{}
\numberwithin{equation}{section}
\maketitle
\begin{abstract}
We consider random Gaussian eigenfunctions of the Laplacian on the standard torus, and investigate the number of nodal intersections against a line segment. The expected intersection number, against any smooth curve, is universally proportional to the length of the reference curve, times the wavenumber, independent of the geometry. We found an upper bound for the nodal intersections variance, depending on whether the slope of the straight line is rational or irrational.
Our findings exhibit a close relation between this problem and the theory of lattice points on circles.
\end{abstract}


\normalfont
\section{Introduction}
\subsection{Nodal intersections and lattice points}
Consider on the torus $\mathbb{T}^2=\mathbb{R}^2/\mathbb{Z}^2$ a real-valued eigenfunction of the Laplacian $F:\mathbb{T}^2\to\mathbb{R}$, with eigenvalue $\lambda^2$:
\begin{equation}
\label{laplacian}
(\Delta +\lambda^2) F=0.
\end{equation}
The nodal set of $F$ is the zero locus
\begin{equation*}
\{x\in\mathbb{T}^2 : F(x)=0\}.
\end{equation*}
Let $\mathcal{C}\subset \mathbb{T}^2$ be a straight line segment on the torus, of length $L$:
\begin{equation*}
\mathcal{C}:\gamma(t)=t\alpha=t(\alpha_1,\alpha_2),
\end{equation*}
with $|\alpha|=1$ and $0\leq t\leq L$. We are interested in the number of \textbf{nodal intersections}
\begin{equation}
\label{Z}
\mathcal{Z}(F)=|\{x : F(x)=0\} \cap \mathcal{C}|,
\end{equation}
the number of zeros of $F$ on $\mathcal{C}$, as $\lambda\to\infty$.

This problem is closely related to the theory of \textbf{lattice points on circles},
as we shall now see. The sequence of Laplace eigenvalues, or energy levels, on $\mathbb{T}^2$ is given by
\begin{equation*}
\{\lambda^2_m=4\pi^2m\}_{m\in S},
\end{equation*}
where $S:=\{m: \ m=a^2+b^2, a,b\in\mathbb{Z}\}$. For $m\in S$, let
\begin{equation}
\label{lpset}
\mathcal{E}=\mathcal{E}_m:=\{\mu\in\mathbb{Z}^2 : |\mu|^2=m\}
\end{equation}
be the set of all lattice points on the circle of radius $\sqrt{m}$. The number $|\mathcal{E}|$ of lattice points equals $r_2(m)$, the number of ways to write $m$ as sum of two integer squares. We shall denote
\begin{equation*}
N=N_m:=|\mathcal{E}|=r_2(m).
\end{equation*}
It is well-known \cite{harwri} that 
$m\in S$ if and only if $m=2^{\nu}\cdot p_1^{\alpha_1}\cdots p_h^{\alpha_h}\cdot q_1^{2\beta_1}\cdots q_l^{2\beta_l}$, where each $p_i\equiv 1 \mod 4$ and each $q_j\equiv 3 \mod 4$; moreover, for $m\in S$,
\begin{equation*}
N_m=4\prod_{i=1}^{h} (\alpha_i+1).
\end{equation*}
Given an eigenvalue $\lambda^2=4\pi^2m$ of \eqref{laplacian}, the collection $\{e^{2\pi i\langle\mu,x\rangle}\}_{\mu\in\mathcal{E}}$ is a basis for the eigenspace. All the eigenfunctions corresponding to the eigenvalue $4\pi^2m$ are
\begin{equation*}
F(x)=
\sum_{\mu\in\mathcal{E}}
c_{\mu}
e^{2\pi i\langle\mu,x\rangle},
\end{equation*}
with $c_{\mu}$ Fourier coefficients. The dimension of the eigenspace is $N_m=r_2(m)$.

\subsection{The model and prior results}
We consider the random Gaussian toral eigenfunctions, called ``arithmetic random waves" \cite{krkuwi}
\begin{equation}
\label{arw}
F(x)=\frac{1}{\sqrt{N_m}}
\sum_{\mu\in\mathcal{E}}
a_{\mu}
e^{2\pi i\langle\mu,x\rangle},
\end{equation}
where $a_{\mu}$ are complex standard Gaussian random variables (i.e. $\mathbb{E}(a_{\mu})=0$ and $\mathbb{E}(|a_{\mu}|^2)=1$), independent save for the relations $a_{-\mu}=\overline{a_{\mu}}$ (so that $F(x)$ is real-valued). 

One is interested in the
distribution of the nodal intersections \eqref{Z}. Rudnick and Wigman \cite{rudwig} computed the expected number of nodal intersections against smooth curves $\mathcal{C}$ of length $L$ on the torus to be
\begin{equation}
\label{expect}
\mathbb{E}[\mathcal{Z}]=\sqrt{2m}L.
\end{equation}
Moreover, they gave precise asymptotics for the variance of $\mathcal{Z}$ against smooth curves with {\em nowhere zero curvature} $\mathcal{C}$ (assuming w.l.o.g. to have unit speed parametrisation $\gamma: [0,L]\to\mathcal{C}$):
\begin{equation}
\label{precasym}
\text{Var}(\mathcal{Z})=
(4B_{\mathcal{C}}(\mathcal{E})-L^2)
\cdot
\frac{m}{N_m}
+
O\bigg(\frac{m}{N_m^{3/2}}\bigg)
\end{equation}
where
\begin{equation*}
B_{\mathcal{C}}(\mathcal{E}):=
\int_{\mathcal{C}}
\int_{\mathcal{C}}
\frac{1}{N_m}
\sum_{\mu\in\mathcal{E}}
\bigg\langle\frac{\mu}{|\mu|},\dot{\gamma}(t_1)\bigg\rangle^2
\cdot
\bigg\langle\frac{\mu}{|\mu|},\dot{\gamma}(t_2)\bigg\rangle^2
dt_1dt_2.
\end{equation*}
This asymptotic behaviour is non-universal: $B_{\mathcal{C}}(\mathcal{E})$ depends both on $\mathcal{C}$ and on the angular distribution of the lattice points. It also follows that the normalised number of nodal intersections $\frac{\mathcal{Z}}{\sqrt{m}}$ is a r.v. with constant mean and vanishing variance (as $m\to\infty$ along a sequence s.t. $N_m\to\infty$): therefore, its distribution is asymptotically concentrated at the mean value.

\subsection{Statement of main results}
We study the nodal intersections $\mathcal{Z}$ for {\em straight} line segments, the other extreme of the nowhere zero curvature setting. Recall that the expectation of $\mathcal{Z}$ is given by \eqref{expect}. 
\begin{thm}
\label{result}
Let $\mathcal{C}\subset\mathbb{T}^2$ be a length $L$ segment with rational slope, i.e. $\gamma(t)=t\alpha$, $\alpha=(\alpha_1,\alpha_2)$ with $\frac{\alpha_2}{\alpha_1}\in\mathbb{Q}$, $|\alpha|=1$, and $\{m\}\subseteq S$ a sequence such that $N_m\rightarrow\infty$. Then
\begin{equation*}
\text{Var}(\mathcal{Z})
=
O
\bigg(
\dfrac{m}{N_m}
\bigg),
\end{equation*}
the implied constant depending on $\alpha$ only.
\end{thm}
\noindent
This upper bound for the variance is the same order of magnitude as the leading term in \eqref{precasym} for the case of nowhere zero curvature.
\\
Without the assumption of rational slope we may prove the following result unconditionally.
\begin{thm}
\label{results}
Let $\mathcal{C}$ be a segment on the torus, and $\{m\}\subseteq S$ a sequence such that $N_m\rightarrow\infty$. Then
\begin{equation*}
\text{Var}(\mathcal{Z})
=
O
\bigg(m\bigg(
\frac{\log m}{N_m}
\bigg)^\frac{4}{5}\bigg)
.
\end{equation*}
\end{thm}
\noindent
The variance of $\frac{\mathcal{Z}}{\sqrt{m}}$ vanishes 
for all sequences $\{m\}\subseteq S$ satisfying
\begin{equation*}
\log m=o(N_m).
\end{equation*}
Examples of such sequences include increasing products of distinct primes
\begin{equation*}
m_k=\prod_{\substack{p\leq k\\p\equiv 1 \mod 4}} p
\end{equation*}
or increasing products of any bounded number of primes (at least two of them), for example
\begin{equation*}
m_k=(5\cdot 13)^k.
\end{equation*}

We may improve the bound of Theorem \ref{results} conditionally on a conjecture about lattice points on short arcs. Consider a circle of radius $R=\sqrt{m}$. It was proven by Jarnik \cite{jarnik} that on every arc of length $<(\sqrt{m})^\frac{1}{3}$ there are at most $2$ lattice points.
Theorem \ref{resultc} below is conditional on a weaker version of a conjecture by Cilleruelo and Granville (Conjecture \ref{conjcg} in Section \ref{lposa}; see also \cite{cilgr2}, \cite{cilgr1}).
\begin{conj}
\label{myconja}
There exists $\epsilon>0$ such that on a circle of radius $R=\sqrt{m}$, on any arc of length $(\sqrt{m})^{\frac{1}{2}+\epsilon}$ there are $O(1)$ lattice points.
\end{conj}
\begin{thm}
\label{resultc}
Assume Conjecture \ref{myconja}. Let $\mathcal{C}$ be a segment on the torus, and $\{m\}\subseteq S$ a sequence such that $N_m\rightarrow\infty$. Then
\begin{equation*}
\text{Var}(\mathcal{Z})
=
O
\bigg(\frac{m}{N_m}\bigg)
.
\end{equation*}
\end{thm}

\noindent
Furthermore, we may prove the bound of Theorem \ref{resultc} unconditionally for a {\em density one sequence} of energy levels (cf. Lemma \ref{BRAHP}).
\begin{thm}
\label{resulta}
Let $\mathcal{C}$ be a segment on the torus, and $\{m\}\subseteq S$ a sequence such that $N_m\rightarrow\infty$ and
\begin{equation*}
\min_{\mu\neq\mu'\in\mathcal{E}_m}|\mu-\mu'|>(\sqrt{m})^{1-\epsilon}
\end{equation*}
for some $0<\epsilon<\frac{1}{2}$ and sufficiently big $m$. Then
\begin{equation*}
\text{Var}(\mathcal{Z})
=
O
\bigg(\frac{m}{N_m}\bigg)
.
\end{equation*}
\end{thm}

\subsection{Outline of the paper}
The rest of this work focuses on proving the stated theorems. In Section \ref{krsection}, thanks to the work of Rudnick and Wigman \cite{rudwig} for generic curves $\mathcal{C}$, we reduce the problem of studying the variance to bounding the second moment of the covariance function $r(t_1,t_2)=\mathbb{E}[F(\gamma(t_1))F(\gamma(t_2))]$ (see \eqref{rgen} below) and a couple of its derivatives. Next, using the hypothesis that $\mathcal{C}$ is a segment, we further reduce our problem to bounding sums over the lattice points. This relies on estimates for the second moment (established in Section \ref{auxpfs}).

There are marked differences compared to the case of generic curves: first, the covariance function has the special form \eqref{r} if $\mathcal{C}$ is a line segment, so that the {\em process} $f(t)=F(\gamma(t))$ (see \eqref{effegen} below) is {\em stationary}. This leads to a different method from \cite{rudwig} of controlling the second moment, and specifically the off-diagonal terms of \eqref{splitsum2}. Indeed, in \cite{rudwig}, Lemma 5.2, the off-diagonal terms are handled via Van der Corput's lemma, applicable for curves $\mathcal{C}$ of nowhere vanishing curvature, whereas the special form \eqref{r} of the covariance function allows us to establish the estimate \eqref{lessthanfrac} directly; the latter term 
happens to be of different nature than the corresponding expression in the non-vanishing curvature case (cf. \cite{rudwig}, Equation (5.18)). This leads to bounding a certain summation over the lattice points, different from \cite{rudwig}: Rudnick and Wigman proved that (see \cite{rudwig}, Proposition 5.3)
\begin{equation*}
\sum_{\substack
{\mu,\mu'\in\mathcal{E}\\\mu\neq\mu'}
}\frac{1}{|\mu-\mu'|}\ll N_m^\epsilon, \quad \forall\epsilon>0,
\end{equation*}
whereas in this work, we need to bound
\begin{equation}
\label{arithprob}
\sum_{
\substack
{\mu,\mu'\in\mathcal{E}\\\langle\mu-\mu',\alpha\rangle\neq 0}
}\frac{1}{\langle\mu-\mu',\alpha\rangle^2}
\end{equation}
where $\alpha$ is the direction of our straight line. In Section \ref{rational}, we bound \eqref{arithprob} for $\alpha$ rational, and complete the proof of Theorem \ref{result}; in Section \ref{irrational}, we treat the case of irrational slope, and complete the proofs of Theorems \ref{results}, \ref{resultc} and \ref{resulta}, following necessary background on the number of lattice points belonging to a short arc of a circle, covered in Section \ref{lposa}.

\section{An approximate Kac-Rice formula}
\label{krsection}
The random Gaussian toral eigenfunction \eqref{arw} is a stationary Gaussian random field. Indeed, the covariance function is
\begin{equation*}
r_F(x,y)
:=
\mathbb{E}[F(x)\cdot F(y)]
=
\frac{1}{N_m}\sum_{\mu\in\mathcal{E}} e^{2\pi i\langle\mu,(x-y)\rangle},
\end{equation*}
depending on $x-y$ only. The covariance function of a random field is non-negative definite (see \cite{cralea}, \S 5.1); a (centred) Gaussian random field is completely determined by its covariance function (see Kolmogorov's Theorem \cite{cralea}, \S 3.3).
\\
For now we assume $\mathcal{C}$ to be a smooth toral curve (which may or may not be a segment). Let $\gamma(t): [0,L]\to\mathbb{T}^2$ be its arc-length parametrisation. We restrict $F$ along $\mathcal{C}$, which yields the (centred Gaussian) random process $f$ on the interval $[0,L]$:
\begin{equation}
\label{effegen}
f(t)=F(\gamma(t))=
\frac{1}{\sqrt{N_m}}
\sum_{\mu\in\mathcal{E}}
a_{\mu}
e^{2\pi i\langle\mu,\gamma(t)\rangle}.
\end{equation}
Its covariance function is
\begin{equation}
\label{rgen}
r(t_1,t_2)
=
\frac{1}{N_m}\sum_{\mu\in\mathcal{E}} e^{2\pi i\langle\mu,\gamma(t_1)-\gamma(t_2)\rangle}.
\end{equation}
The quantity we are studying, i.e. the number of nodal intersections $\mathcal{Z}$, equals the number of zero crossings of the process $f$ (on $[0,L]$).
The moments of a random variable that counts the number of crossings of a level by a process $f:I\to\mathbb{R}$ are given by the \textbf{Kac-Rice formulas} (see \cite{cralea}, \S 10, and \cite{azawsc}, Theorem 3.2). For each $t$, let $\phi_{f(t)}$ be the probability density function of the (standard Gaussian) random variable $f(t)$, and $\phi_{f(t_1),f(t_2)}$ the joint density of the random vector $(f(t_1),f(t_2))$. We define the \textbf{zero density} function $K_1: [0,L]\to\mathbb{R}$ and \textbf{2-point correlation} function $K_2: [0,L]\times[0,L]\to\mathbb{R}$ of a process $f$ as the Gaussian expectations 
\begin{equation*}
K_1(t)=\phi_{f(t)}(0)\cdot\mathbb{E}[|f'(t)|\big| f(t)=0]
\end{equation*}
\begin{equation*}
K_2(t_1,t_2)=\phi_{f(t_1),f(t_2)}(0,0)\cdot\mathbb{E}[|f'(t_1)|\cdot|f'(t_2)|\big| f(t_1)=f(t_2)=0],
\end{equation*}
the latter defined for $t_1\neq t_2$. The Kac-Rice formulas for the first and second (factorial) moments of the number of crossings are
\begin{equation}
\label{kacrice1}
\mathbb{E}(\mathcal{Z})=
\int_{0}^{L}K_1(t)dt,
\end{equation}
\begin{equation}
\label{kacrice2}
\mathbb{E}(\mathcal{Z}^2-\mathcal{Z})=
\int_{0}^{L}\int_{0}^{L}K_2(t_1,t_2)dt_1dt_2.
\end{equation}
Rudnick and Wigman proved that $K_1(t)\equiv\sqrt{2}\sqrt{m}$ (see \cite{rudwig}, Lemma 2.1), and via \eqref{kacrice1} they computed the expected intersection number (recall \eqref{expect}).
\\
The Kac-Rice formula for the second moment \eqref{kacrice2} holds provided the following non-degeneracy condition is met by $f$:
the centred Gaussian distribution of the vector $(f(t_1),f(t_2))$ must be nondegenerate for all $(t_1,t_2)\in[0,L]\times[0,L]$ such that $t_1\neq t_2$ (see \cite{azawsc}, \S 3). This may fail for $f$ as in (\ref{effegen});
however, Rudnick and Wigman developed an \textbf{approximate Kac-Rice formula}. Denote 
\begin{equation*}
r_1=\frac{\partial r(t_1,t_2)}{\partial t_1},
\qquad
r_2=\frac{\partial r(t_1,t_2)}{\partial t_2}
\qquad
\text{and}
\quad
r_{12}=\frac{\partial^2 r(t_1,t_2)}{\partial t_1\partial t_2}
\end{equation*}
the derivatives of the covariance function \eqref{rgen}. 
\begin{prop}[\emph{Approximate Kac-Rice bound} \cite{ruwiye}, Proposition 2.2]
\label{approxKR}
We have
\begin{equation*}
\text{Var}(\mathcal{Z})
=
m
\cdot
O
(\mathcal{R}_2(m))
\end{equation*}
where
\begin{equation}
\label{2ndmom}
\mathcal{R}_2(m)
:=
\int_0^L \int_0^L 
\bigg(
r^2+\bigg(\frac{r_1}{\sqrt{m}}\bigg)^2+\bigg(\frac{r_2}{\sqrt{m}}\bigg)^2+\bigg(\frac{r_{12}}{m}\bigg)^2
\bigg)
dt_1dt_2
.
\end{equation}
\end{prop}
\noindent
This result is applicable to the case where $\mathcal{C}$ is a segment, as it holds for all smooth curves. Note that the approximate Kac-Rice formula \cite{rudwig}, Proposition 1.3 gives both the leading term and the error term for the variance; the upper bound of Proposition \ref{approxKR} is sufficient for our purposes. Our problem is thus reduced to bounding the second moment of the covariance function and a couple of its derivatives along $\mathcal{C}$.


From this point on, assume $\mathcal{C}\subset \mathbb{T}^2$ to be a segment; we write
\begin{equation*}
\gamma(t)=t\alpha=t(\alpha_1,\alpha_2),
\end{equation*}
with $|\alpha|=1$ and $0\leq t\leq L$. In this case, \eqref{effegen} becomes
\begin{equation*}
\label{effe}
f(t)=
\frac{1}{\sqrt{N_m}}
\sum_{\mu\in\mathcal{E}}
a_{\mu}
e^{2\pi it\langle\mu,\alpha\rangle}
\end{equation*}
and the covariance function of the process is
\begin{equation}
\label{r}
r(t_1,t_2)
=
\frac{1}{N_m}\sum_{\mu\in\mathcal{E}} e^{2\pi i(t_1-t_2)\langle\mu,\alpha\rangle},
\end{equation}
depending on the difference $t_1-t_2$ only. Therefore, if $\mathcal{C}$ is a segment, then the process $f$ is {\em stationary} (and without loss of generality we may assume that $\mathcal{C}$ contains the origin).
\\
We now further reduce our problem to bounding a sum over the lattice points.
\begin{defin}
\label{setA}
Given a nonzero vector $v\in\mathbb{R}^2$, we define the set
\begin{equation*}
A_v:=
\{(\mu,\mu')\in\mathcal{E}^2 : \langle\mu-\mu',v\rangle\neq 0\}
\end{equation*}
with $\mathcal{E}$ as in \eqref{lpset}.
\end{defin}
\begin{prop}
\label{mainprop}
Assuming $\mathcal{C}$ to be a segment,
\begin{equation*}
\text{Var}(\mathcal{Z})
\ll
\frac{m}{N_m}
+
\frac{m}{N_m^2}
\cdot
\sum_{A_\alpha}
\min\bigg(1,\frac{1}{\langle\mu-\mu',\alpha\rangle^2}\bigg).
\end{equation*}
\end{prop}
\noindent
The proof of Proposition \ref{mainprop} is given in Section \ref{auxpfs}.

\section{The case of rational lines}
\label{rational}
The goal of this section is to prove Theorem \ref{result}. Recall that
\begin{equation*}
\mathcal{E}=\{\mu\in\mathbb{Z}^2 : |\mu|^2=m\}
\end{equation*}
is the set of lattice points lying on the circle of radius $\sqrt{m}$, and $N_m=|\mathcal{E}|$ is their number. By Proposition \ref{mainprop}, it is sufficient to bound the summation
\begin{equation*}
\sum_{A_\alpha}\frac{1}{\langle\mu-\mu',\alpha\rangle^2}.
\end{equation*}
\begin{prop}
\label{l1}
Let $\alpha=(\alpha_1,\alpha_2)$ with $\frac{\alpha_2}{\alpha_1}\in\mathbb{Q}$, and $A_\alpha$ be as in Definition \ref{setA}. Then
\begin{equation}
\label{2state}
\sum_{A_\alpha}\frac{1}{\langle\mu-\mu',\alpha\rangle^2}\ll_{\alpha} N_m.
\end{equation}
\end{prop}
\begin{proof}
Up to multiplication by a scalar, $\alpha$ has integer coordinates:
\begin{equation*}
\alpha=
(\alpha_1,\alpha_2)=
\alpha_1\bigg(1,\frac{\alpha_2}{\alpha_1}\bigg)=
\alpha_1\bigg(1,\frac{p}{q}\bigg)=
\frac{\alpha_1}{q}\cdot(q,p)
\end{equation*}
for some $p,q\in\mathbb{Z}$ and $q\neq 0$. Note that
$
A_\alpha
=
A_{(q,p)}
$
because the vectors $\alpha$ and $(q,p)$ are collinear. It follows that
\begin{equation}
\label{pq}
\sum_{A_\alpha}\frac{1}{\langle\mu-\mu',\alpha\rangle^2}
=
\frac{q^2}{\alpha_1^2}
\cdot
\sum_{A_{(q,p)}}\frac{1}{\langle\mu-\mu',(q,p)\rangle^2}
\ll_{\alpha}
\sum_{A_{(q,p)}}\frac{1}{\langle\mu-\mu',(q,p)\rangle^2}.
\end{equation}
Next, let $\mu$ be fixed, and consider $k=\langle\mu-\mu',(q,p)\rangle$. As both $\mu-\mu'$ and $(q,p)$ have integer coordinates, $k\in\mathbb{Z}$; moreover, as $(\mu,\mu')\in A_{(q,p)}$, $k\neq 0$. Then
\begin{equation}
\label{triplesum}
\sum_{A_{(q,p)}}\frac{1}{\langle\mu-\mu',(q,p)\rangle^2}
=
\sum_{\mu}\sum_{k\neq 0}\sum_{\substack{\mu' \\ \langle\mu-\mu',(q,p)\rangle=k}}\frac{1}{k^2}.
\end{equation}
We now show that there can be at most two terms in the inner-most summation: the lattice point $\mu'$ of the circle $x^2+y^2=m$ has to satisfy, for fixed $\mu$ and $k$,
\begin{equation*}
\langle\mu',(q,p)\rangle=\langle\mu,(q,p)\rangle-k=\mu_1q+\mu_2p-k=:h.
\end{equation*}
Thus $\mu'$ is lying on the straight line $qx+py=h$, and a circle and a line can intersect in at most two points. Therefore,
\begin{equation}
\label{pisquaredover3}
\sum_{\mu}\sum_{k\neq 0}\sum_{\substack{\mu' \\ \langle\mu-\mu',(q,p)\rangle=k}}\frac{1}{k^2}
\leq
2\sum_{\mu}\sum_{k\neq 0}\frac{1}{k^2}
=
2\cdot\frac{\pi^2}{3}N_m
\ll
N_m.
\end{equation}
Combining \eqref{pq}, \eqref{triplesum} and \eqref{pisquaredover3} we get the statement \eqref{2state} of Proposition \ref{l1}.
\end{proof}
\begin{proof}[Proof of Theorem \ref{result}]
Applying Proposition \ref{mainprop}, we have
\begin{equation}
\label{applymainprop}
\text{Var}(\mathcal{Z})
\ll
\frac{m}{N_m}
+
\frac{m}{N_m^2}
\cdot
\sum_{A_\alpha}
\min\bigg(1,\frac{1}{\langle\mu-\mu',\alpha\rangle^2}\bigg)
\end{equation}
with $A_\alpha$ as in Definition \ref{setA}. By Proposition \ref{l1},
\begin{equation*}
\sum_{A_\alpha}
\min\bigg(1,\frac{1}{\langle\mu-\mu',\alpha\rangle^2}\bigg)\ll N_m
\end{equation*}
and the statement of Theorem \ref{result} follows.
\end{proof}

\section{Lattice points on short arcs}
\label{lposa}
The number of lattice points $N_m$ on the circle of radius $\sqrt{m}$ has the upper bound
\begin{equation}
\label{Nissmall}
N_m\ll m^\epsilon \qquad \forall\epsilon>0,
\end{equation}
and the analogous statement with powers of logarithms of $m$ in place of $m^\epsilon$ is false \cite{harwri}. We are interested in upper bounds for the number of lattice points on short arcs of
the circle (the term `short' indicates that
the length of the arc is small compared to the radius): we now review the known bounds. As mentioned in the introduction, on any arc of length $<(\sqrt{m})^\frac{1}{3}$ of the circle there are at most $2$ lattice points \cite{jarnik}.
\\
Moreover, Cilleruelo and C\'{o}rdoba \cite{cilcor} proved that, for all integers $l\geq 1$, on any arc of length
$\leq \sqrt{2}(\sqrt{m})^
{\frac{1}{2}-\frac{1}{(4\lfloor \frac{l}{2} \rfloor+2)}}$
there are at most $l$ lattice points.
\begin{prop}[Bourgain and Rudnick \cite{bourud}, Lemma 2.1]
\label{log}
On any arc of length at most $(\sqrt{m})^{\frac{1}{2}}$ of a circle of radius $\sqrt{m}$, there are $O(\log{m})$ lattice points.
\end{prop}
\begin{conj}[Cilleruelo and Granville \cite{cilgr2}, \cite{cilgr1}]
\label{conjcg}
Consider a circle of radius $R=\sqrt{m}$. For all $\delta>0$, there exists a constant $C_{\delta}$ such that on any arc of length $(\sqrt{m})^{1-\delta}$ there are at most $C_{\delta}$ lattice points.
\end{conj}
\noindent
Conjecture \ref{conjcg} implies Conjecture \ref{myconja}.
Furthermore, Bourgain and Rudnick \cite{brahpo} showed that Conjecture \ref{conjcg} is true 
for `most' $m\in S$. Recalling that $S=\{m: \ m=a^2+b^2, a,b\in\mathbb{Z}\}$, define
\begin{equation*}
S(X):=\{m\in S : m\leq X\}.
\end{equation*}
It is known \cite{landau} that, as $X\to\infty$,
$S(X)\sim C\frac{X}{\sqrt{\log X}}$,
where $C>0$ is the Landau-Ramanujan constant.
\begin{lemma}[Bourgain and Rudnick \cite{brahpo}, Lemma 5]
\label{BRAHP}
Fix $\epsilon>0$. Then for all but $O(X^{1-\frac{\epsilon}{3}})$ integers $m\leq X$, one has
\begin{equation*}
\min_{\mu\neq\mu'\in\mathcal{E}}|\mu-\mu'|>(\sqrt{m})^{1-\epsilon}.
\end{equation*}
\end{lemma}
\noindent
Therefore, the assumptions of Theorem \ref{resulta} hold for a {\em density one sequence} of energy levels.


\section{The case of irrational lines}
\label{irrational}
The goal of this section is to prove Theorems \ref{results}, \ref{resultc} and \ref{resulta}.
\subsection{Preparatory results}
Denote $\sqrt{m}\mathcal{S}^1$ the radius $\sqrt{m}$ circle.
\begin{lemma}
\label{arcDE}
Let $c=c(m)>0$, with $c\to 0$ as $m\to\infty$. Fix a point $B\in\sqrt{m}\mathcal{S}^1$, and let $\beta$ be a unit vector.  
Then there exists an arc $\overset{\frown}{DE}$ of $\sqrt{m}\mathcal{S}^1$ of length $(4c+O(c^3))\sqrt{m}$ such that all points $B'\in\sqrt{m}\mathcal{S}^1$ satisfying $B'\neq B$ and $|\langle B-B',\beta\rangle|\leq c|B-B'|$ lie on $\overset{\frown}{DE}$.
\end{lemma}
\begin{proof}
The condition $|\langle B-B',\beta\rangle|\leq c|B-B'|$ means $B-B'$ and $\beta$ are close to being orthogonal, in the sense that $|\cos(\varphi_{B-B',\beta})|\leq c$, where $0\leq\varphi_{v,w}\leq\pi$ denotes the angle between two non-zero vectors $v,w\in\mathbb{R}^2$. Let $s',s''$ be the two straight lines through $B$ satisfying
\begin{equation*}
\label{condmu'1}
|\cos(\varphi_{s',\beta})|
=
|\cos(\varphi_{s'',\beta})|
=c.
\end{equation*}
Let $D$ be the further intersection between the circle $\sqrt{m}\mathcal{S}^1$ and $s'$, meaning $\sqrt{m}\mathcal{S}^1\cap s'=\{B,D\}$. Likewise, let $E$ be the further intersection between $\sqrt{m}\mathcal{S}^1$ and $s''$, meaning $\sqrt{m}\mathcal{S}^1\cap s''=\{B,E\}$. Note that possibly one of the lines $s',s''$, say $s''$, is tangent to the circle $\sqrt{m}\mathcal{S}^1$, in which case $E=B$. We have $B'\in\overset{\frown}{DE}$ and we want to show $\overset{\frown}{DE}=(4c+O(c^3))\sqrt{m}$.
\\
By the expansion
\begin{equation*}
\arccos(c)=\frac{\pi}{2}-c+O(c^3) 
\end{equation*}
we have
\begin{equation*}
\varphi_{s',\beta}
=
\varphi_{s'',\beta}
=
\frac{\pi}{2}-c+O(c^3),
\qquad
\varphi_{s',s''}
=
\pi-\varphi_{s',\beta}-\varphi_{s'',\beta}=2c+O(c^3).
\end{equation*}
Let $D',D''$ be points on $s'$ on opposite sides of $B$, and $E',E''$ be points on $s''$ on opposite sides of $B$, so that:
$\overline{BD'}=\overline{BD''}=\overline{BE'}=\overline{BE''}=3\sqrt{m}$,
$D$ lies on $s'$ between $B$ and $D'$, and
$\widehat{D'BE'}=\varphi_{s',s''}=2c+O(c^3)$.
There are three cases:
\begin{itemize}
\item
In case $E$ lies on $s''$ between $B$ and $E'$, we have
\begin{equation*}
\overset{{\displaystyle \frown}}{DE}=\widehat{DOE}\cdot\sqrt{m}=2\widehat{D'BE'}\cdot\sqrt{m}=(4c+O(c^3))\sqrt{m}
\end{equation*}
where we have denoted $O$ the origin, centre of $\sqrt{m}\mathcal{S}^1$.
\item
In case $E$ lies on $s''$ between $B$ and $E''$, then $B$ lies on the arc $\overset{\frown}{DE}$ and we have
\begin{gather*}
\overset{{\displaystyle \frown}}{DE}=(\widehat{DOB}+\widehat{EOB})\sqrt{m}=(2\widehat{DEB}+2\widehat{EDB})\sqrt{m}=2\widehat{D'BE'}\cdot\sqrt{m}
\\
=(4c+O(c^3))\sqrt{m}.
\end{gather*}
\item
In case $E=B$, we write
\begin{equation*}
\overset{{\displaystyle \frown}}{DE}
=
\overset{{\displaystyle \frown}}{DB}
=
\widehat{DOB}\cdot\sqrt{m}
=
2\widehat{D'BE'}\cdot\sqrt{m}
=(4c+O(c^3))\sqrt{m}.
\end{equation*}
\end{itemize}
\end{proof}

\noindent
For two functions $f(m), g(m)$, we write $f\sim g$ if, as $m\to\infty$, the ratio of the two sides converges to $1$.
\begin{prop}
\label{generalise}
Let $A_\alpha$ be as in Definition \ref{setA}, and recall that $|\alpha|=1$.
Assume that every arc on $\sqrt{m}\mathcal{S}^1$ of length $J$ contains at most $l$ lattice points.
Then
\begin{equation*}
\frac{1}{N_m^2}
\cdot
\sum_{A_\alpha}\min\bigg(1,\frac{1}{\langle\mu-\mu',\alpha\rangle^2}\bigg)
\ll
\bigg(\bigg(
\frac{l}{J}
\bigg)^4
\cdot
\frac
{m}
{{N_m}^4}
\bigg)^\frac{1}{5}
+
\frac{l}{N_m}
.
\end{equation*}
\end{prop}
\begin{proof}
Let $a\leq 2\sqrt{m}$ and $c$ be positive parameters, such that $c\to 0$ as $m\to\infty$. We separate the sum over the following three ranges:
\begin{itemize}
\item
first range: $|\mu-\mu'|\leq a$
\item
second range: $|\langle\mu-\mu',\alpha\rangle|\leq c|\mu-\mu'|$
\item
third range: $|\mu-\mu'|\geq a, \ 
|\langle\mu-\mu',\alpha\rangle|\geq c|\mu-\mu'|$.
\end{itemize}
We may now rewrite
\begin{multline}
\label{splitsum}
\sum_{A_\alpha}\min\bigg(1,\frac{1}{\langle\mu-\mu',\alpha\rangle^2}\bigg)
\leq
\#\{(\mu,\mu') : |\mu-\mu'|\leq a\}
\\+
\#\{(\mu,\mu') : |\langle\mu-\mu',\alpha\rangle|\leq c|\mu-\mu'|\}
+
\sum_{
\substack{|\mu-\mu'|\geq a
\\
|\langle\mu-\mu',\alpha\rangle|\geq c|\mu-\mu'|}}
\frac{1}{\langle\mu-\mu',\alpha\rangle^2}
.
\end{multline}
\underline{First range}: recall the notation $\sqrt{m}\mathcal{S}^1$ for the radius $\sqrt{m}$ circle. For a fixed lattice point $\mu$, all $\mu'$ satisfying $|\mu-\mu'|\leq a$ must lie on a disc centred at $\mu$ with radius $a$; the intersection of this disc with $\sqrt{m}\mathcal{S}^1$ is an arc on $\sqrt{m}\mathcal{S}^1$ of length $\sim a$ around $\mu$. To bound (from above) the number of $\mu'$ on this arc, we partition it into small arcs of length $J$: there are $\ll 1+\frac{a}{J}$ small arcs, and by the assumptions of Proposition \ref{generalise} each contains at most $l$ lattice points. Therefore,
\begin{equation}
\label{firstrange}
\#\{(\mu,\mu') : |\mu-\mu'|\leq a\}
=
O\bigg(\frac{a}{J}\cdot\l\cdot N_m\bigg)+O(l\cdot N_m).
\end{equation}
\underline{Second range}: fix a lattice point $\mu$ and apply Lemma \ref{arcDE} with $\beta=\alpha$. Then all $\mu'$ satisfying $|\langle\mu-\mu',\alpha\rangle|\leq c|\mu-\mu'|$ must lie on an arc of length $(4c+O(c^3))\sqrt{m}$ on the circle $\sqrt{m}\mathcal{S}^1$. Partition this arc into small arcs of length $J$: there are $\ll 1+\frac{4c\sqrt{m}}{J}$ small arcs, and each contains at most $l$ lattice points. It follows that
\begin{equation}
\label{secondrange}
\#\{(\mu,\mu') : |\langle\mu-\mu',\alpha\rangle|\leq c|\mu-\mu'|\}
=
O\bigg(\frac{c\sqrt{m}}{J}\cdot l\cdot N_m\bigg)+O(l\cdot N_m).
\end{equation}
\underline{Third range}. Here we have $|\mu-\mu'|\geq a$ and $|\langle\mu-\mu',\alpha\rangle|\geq c|\mu-\mu'|$, therefore
\begin{equation}
\label{thirdrange}
\sum
\frac{1}{\langle\mu-\mu',\alpha\rangle^2}
\leq
\sum
\frac{1}{(\mu-\mu')^2c^2}
\leq
\sum
\frac{1}{a^2c^2}
\leq
\frac{N_m^2}{a^2c^2}.
\end{equation}
Substituting \eqref{firstrange}, \eqref{secondrange} and \eqref{thirdrange} into \eqref{splitsum}, we obtain
\begin{multline*}
\sum_{A_\alpha}\min\bigg(1,\frac{1}{\langle\mu-\mu',\alpha\rangle^2}\bigg)
\\=
O\bigg(\frac{a}{J}\cdot\l\cdot N_m\bigg)
+
O\bigg(\frac{c\sqrt{m}}{J}\cdot l\cdot N_m\bigg)
+
O(l\cdot N_m)
+
O\bigg(\frac{N_m^2}{a^2c^2}\bigg).
\end{multline*}
The optimal choices for the parameters are
\begin{equation*}
a=c\sqrt{m}=\bigg(\frac{J}{l}\bigg)^{\frac{1}{5}}
\cdot
{N_m}^{\frac{1}{5}}
\cdot
m^{\frac{1}{5}},
\end{equation*}
and it follows that
\begin{equation*}
\frac{1}{N_m^2}
\cdot
\sum_{A_\alpha}\min\bigg(1,\frac{1}{\langle\mu-\mu',\alpha\rangle^2}\bigg)
\ll
\bigg(
\frac{l}{J}
\bigg)^{\frac{4}{5}}
\frac
{m^{\frac{1}{5}}}
{N_m^\frac{4}{5}}
+
\frac{l}{N_m}
.
\end{equation*}
\end{proof}

\subsection{Proofs of Theorems \ref{results}, \ref{resultc} and \ref{resulta}}
\begin{cor}
\label{smom}
We have unconditionally
\begin{equation*}
\frac{1}{N_m^2}
\cdot
\sum_{A_\alpha}\min\bigg(1,\frac{1}{\langle\mu-\mu',\alpha\rangle^2}\bigg)
\ll
\bigg(\frac{\log m}{N_m}\bigg)^{\frac{4}{5}}
+
\frac{\log m}{N_m}
.
\end{equation*}
\end{cor}
\begin{proof}
By Proposition \ref{log}, we may take $J=(\sqrt{m})^{\frac{1}{2}}$ and $l=O(\log(m))$ unconditionally in Proposition \ref{generalise}.
\end{proof}
\begin{proof}[Proof of Theorem \ref{results}]
Apply Proposition \ref{mainprop}, yielding \eqref{applymainprop}; by Corollary \ref{smom}, we have
\begin{equation}
\label{latterexpression}
\text{Var}(\mathcal{Z})
\ll
\frac{m}{N_m}
+
m
\cdot
\bigg(\frac{\log m}{N_m}\bigg)^{\frac{4}{5}}
+
m
\cdot
\frac{\log m}{N_m}
\ll
m
\cdot
\bigg(\frac{\log m}{N_m}\bigg)^{\frac{4}{5}}
\end{equation}
where we have assumed $\log m=o(N_m)$. 
\end{proof}

\begin{cor}
\label{cond}
Assume Conjecture \ref{myconja}. Then
\begin{equation*}
\frac{1}{N_m^2}
\cdot
\sum_{A_\alpha}\min\bigg(1,\frac{1}{\langle\mu-\mu',\alpha\rangle^2}\bigg)
\ll
\frac{1}{N_m}.
\end{equation*}
\end{cor}
\begin{proof}
By Conjecture \ref{myconja}, for some $\epsilon>0$, we may take $J=(\sqrt{m})^{\frac{1}{2}+\epsilon}$ and $l=O(1)$ in Proposition \ref{generalise}:
\begin{equation*}
\frac{1}{N_m^2}
\cdot
\sum_{A_\alpha}\min\bigg(1,\frac{1}{\langle\mu-\mu',\alpha\rangle^2}\bigg)
\ll
\bigg(\bigg(
\frac{1}{m^{\frac{1}{4}+\frac{\epsilon}{2}}}
\bigg)^4
\cdot
\frac
{m}
{{N_m}^4}
\bigg)^\frac{1}{5}
+
\frac{1}{N_m}
\ll
\frac{1}{N_m}
,
\end{equation*}
where the latter inequality follows from \eqref{Nissmall}.
\end{proof}

\begin{proof}[Proof of Theorem \ref{resultc}]
Apply Proposition \ref{mainprop}, yielding \eqref{applymainprop}; by Corollary \ref{cond},
\begin{equation*}
\text{Var}(\mathcal{Z})
\ll
\frac{m}{N_m}.
\end{equation*}
\end{proof}

\begin{cor}
\label{mostm2}
Let $\{m\}\subseteq S$ be a sequence satisfying
\begin{equation*}
\min_{\mu\neq\mu'\in\mathcal{E}_m}|\mu-\mu'|>(\sqrt{m})^{1-\epsilon}
\end{equation*}
for some $0<\epsilon<\frac{1}{2}$ and sufficiently big $m$. Then
\begin{equation*}
\frac{1}{N_m^2}
\cdot
\sum_{A_\alpha}\min\bigg(1,\frac{1}{\langle\mu-\mu',\alpha\rangle^2}\bigg)
\ll
\frac{1}{N_m}
.
\end{equation*}
\end{cor}
\begin{proof}
By the assumptions of Corollary \ref{mostm2},
we have that on the circle $\sqrt{m}\mathcal{S}^1$ on any arc of length $<(\sqrt{m})^{1-\epsilon}$ there is at most one lattice point. Therefore, we may take $J=(\sqrt{m})^{1-\epsilon}$ and $l=1$ in Proposition \ref{generalise}, yielding
\begin{equation*}
\frac{1}{N_m^2}
\cdot
\sum_{A_\alpha}\min\bigg(1,\frac{1}{\langle\mu-\mu',\alpha\rangle^2}\bigg)
\ll
\bigg(\bigg(
\frac{1}{m^{\frac{1}{2}-\frac{\epsilon}{2}}}
\bigg)^4
\cdot
\frac
{m}
{{N_m}^4}
\bigg)^\frac{1}{5}
+
\frac{1}{N_m}
\ll
\frac{1}{N_m}
,
\end{equation*}
where the latter inequality follows from \eqref{Nissmall}.
\end{proof}

\begin{proof}[Proof of Theorem \ref{resulta}]
Apply Proposition \ref{mainprop}, yielding \eqref{applymainprop}; by Corollary \ref{mostm2}, we have
\begin{equation*}
\text{Var}(\mathcal{Z})
\ll
\frac{m}{N_m}
.
\end{equation*}
\end{proof}

\section{The second moment of $r$ and of its derivatives}
\label{auxpfs}
In this section we prove Proposition \ref{mainprop}, for which we need two auxiliary lemmas. Recall that $r=r(t_1,t_2)$ is the covariance function restricted to $\mathcal{C}$, and the notation
\begin{equation*}
r_1=\frac{\partial r(t_1,t_2)}{\partial t_1},
\qquad
r_2=\frac{\partial r(t_1,t_2)}{\partial t_2}
\qquad
\text{and}
\quad
r_{12}=\frac{\partial^2 r(t_1,t_2)}{\partial t_1\partial t_2}.
\end{equation*}
Also recall the definition \eqref{2ndmom} of $\mathcal{R}_2(m)$.
\begin{lemma}
\label{rsq}
Let $\mathcal{C}$ be a segment. Then
\begin{equation*}
\mathcal{R}_2(m)
\ll \frac{1}{N_m^2}
\sum_{(\mu,\mu')\in\mathcal{E}^2}
\bigg|\int_0^L e^{2\pi it\langle\mu-\mu',\alpha\rangle}dt\bigg|^2.
\end{equation*}
\end{lemma}
\begin{proof}
We will show
\begin{equation}
\label{rbound}
\int_0^L \int_0^L 
r^2(t_1,t_2)
dt_1dt_2
\ll\frac{1}{N_m^2}
\sum_{(\mu,\mu')\in\mathcal{E}^2}
\bigg|\int_0^L e^{2\pi it\langle\mu-\mu',\alpha\rangle}dt\bigg|^2,
\end{equation}
\begin{equation}
\label{r1bound}
\int_0^L \int_0^L 
\bigg(\frac{r_i(t_1,t_2)}{\sqrt{m}}\bigg)^2
dt_1dt_2
\ll \frac{1}{N_m^2}
\sum_{(\mu,\mu')\in\mathcal{E}^2}
\bigg|\int_0^L e^{2\pi it\langle\mu-\mu',\alpha\rangle}dt\bigg|^2,
\end{equation}
for $i=1,2$, and
\begin{equation}
\label{r12bound}
\int_0^L \int_0^L 
\bigg(\frac{r_{12}(t_1,t_2)}{m}\bigg)^2
dt_1dt_2
\ll \frac{1}{N_m^2}
\sum_{(\mu,\mu')\in\mathcal{E}^2}
\bigg|\int_0^L e^{2\pi it\langle\mu-\mu',\alpha\rangle}dt\bigg|^2.
\end{equation}
We begin by squaring the covariance function \eqref{r}:
\begin{equation*}
|r|^2
=
\frac{1}{N_m^2}\sum_{(\mu,\mu')\in\mathcal{E}^2}e^{2\pi i(t_1-t_2)\langle\mu-\mu',\alpha\rangle}
\end{equation*}
so that
\begin{gather*}
\int_0^L \int_0^L |r(t_1,t_2)|^2dt_1dt_2
=
\int_0^L \int_0^L
\frac{1}{N_m^2}
\sum_{(\mu,\mu')\in\mathcal{E}^2}
e^{2\pi i(t_1-t_2)\langle\mu-\mu',\alpha\rangle}
dt_1dt_2
\\
=
\frac{1}{N_m^2}
\sum_{(\mu,\mu')\in\mathcal{E}^2}
\int_0^L
e^{2\pi it_1\langle\mu-\mu',\alpha\rangle}
dt_1
\int_0^L
e^{-2\pi it_2\langle\mu-\mu',\alpha\rangle}
dt_2
\\
=
\frac{1}{N_m^2}
\sum_{(\mu,\mu')\in\mathcal{E}^2}
\bigg|\int_0^L e^{2\pi it\langle\mu-\mu',\alpha\rangle}dt\bigg|^2,
\end{gather*}
yielding \eqref{rbound}. Next,
\begin{equation*}
r_1=\frac{\partial r(t_1,t_2)}{\partial t_1}
=
\frac{1}{N_m}\sum_{\mu\in\mathcal{E}}2\pi i\langle\mu,\alpha\rangle e^{2\pi i(t_1-t_2)\langle\mu,\alpha\rangle}
\end{equation*}
and it follows that
\begin{equation*}
\frac{r_1}{2\pi i\sqrt{m}}
=
\frac{1}{N_m}\sum_{\mu\in\mathcal{E}}\bigg\langle\frac{\mu}{|\mu|},\alpha\bigg\rangle e^{2\pi i(t_1-t_2)\langle\mu,\alpha\rangle}.
\end{equation*}
By Cauchy-Schwartz,
\begin{gather*}
\int_0^L \int_0^L \bigg|\frac{r_1}{2\pi\sqrt{ m}}\bigg|^2dt_1dt_2
\\
=
\int_0^L \int_0^L \frac{1}{N_m^2}
\sum_{(\mu,\mu')\in\mathcal{E}^2}
\bigg\langle\frac{\mu}{|\mu|},\alpha\bigg\rangle \bigg\langle\frac{\mu'}{|\mu'|},\alpha\bigg\rangle 
e^{2\pi i(t_1-t_2)\langle\mu,\alpha\rangle}
e^{2\pi i(t_1-t_2)\langle\mu',\alpha\rangle}
dt_1dt_2
\\
\leq
\int_0^L \int_0^L \frac{1}{N_m^2}
\sum_{(\mu,\mu')\in\mathcal{E}^2} 
e^{2\pi i(t_1-t_2)\langle\mu,\alpha\rangle}
e^{2\pi i(t_1-t_2)\langle\mu',\alpha\rangle}
dt_1dt_2
\\
=
\frac{1}{N_m^2}
\sum_{(\mu,\mu')\in\mathcal{E}^2}
\int_0^L
e^{2\pi it_1\langle\mu-\mu',\alpha\rangle}
dt_1
\int_0^L
e^{-2\pi it_2\langle\mu-\mu',\alpha\rangle}
dt_2
\\
=
\frac{1}{N_m^2}
\sum_{(\mu,\mu')\in\mathcal{E}^2} \bigg|\int_0^L e^{2\pi it\langle\mu-\mu',\alpha\rangle}dt\bigg|^2
\end{gather*}
and \eqref{r1bound} follows.
For the second mixed derivative:
\begin{equation*}
r_{12}=\frac{\partial^2 r(t_1,t_2)}{\partial t_1\partial t_2}
=
\frac{1}{N_m}\sum_{\mu\in\mathcal{E}}(2\pi i)^2\langle\mu,\alpha\rangle^2 e^{2\pi i(t_1-t_2)\langle\mu,\alpha\rangle}
\end{equation*}
thus
\begin{equation*}
-\frac{r_{12}}{4\pi^2 m}
=
\frac{1}{N_m}\sum_{\mu\in\mathcal{E}} \bigg\langle\frac{\mu}{|\mu|},\alpha\bigg\rangle^2 e^{2\pi i(t_1-t_2)\langle\mu,\alpha\rangle}.
\end{equation*}
Again by Cauchy-Schwartz,
\begin{gather*}
\int_0^L \int_0^L \bigg|\frac{r_{12}}{4\pi^2 m}\bigg|^2dt_1dt_2
\\
=
\int_0^L \int_0^L \frac{1}{N_m^2}
\sum_{(\mu,\mu')\in\mathcal{E}^2}
\bigg\langle\frac{\mu}{|\mu|},\alpha\bigg\rangle^2 \bigg\langle\frac{\mu'}{|\mu'|},\alpha\bigg\rangle^2 
e^{2\pi i(t_1-t_2)\langle\mu,\alpha\rangle}
e^{2\pi i(t_1-t_2)\langle\mu',\alpha\rangle}
dt_1dt_2
\\
\leq
\frac{1}{N_m^2}
\sum_{(\mu,\mu')\in\mathcal{E}^2}
\bigg|\int_0^L e^{2\pi it\langle\mu-\mu',\alpha\rangle}dt\bigg|^2,
\end{gather*}
yielding \eqref{r12bound}.
\end{proof}

\begin{lemma}
We have the following bound:
\label{sumlp}
\begin{equation*}
\sum_{(\mu,\mu')\in\mathcal{E}^2}
\bigg|\int_0^L e^{2\pi it\langle\mu-\mu',\alpha\rangle}dt\bigg|^2
\ll
N_m
+
\sum_{A_\alpha}
\min\bigg(1,\frac{1}{\langle\mu-\mu',\alpha\rangle^2}\bigg).
\end{equation*}
\end{lemma}
\begin{proof}
We split the summation over three ranges: diagonal pairs, off-diagonal pairs satisfying $\mu-\mu'\perp\alpha$, and the set $A_\alpha$ of Definition \ref{setA}:
\begin{multline}
\label{splitsum2}
\sum_{(\mu,\mu')\in\mathcal{E}^2}
\bigg|\int_0^L e^{2\pi it\langle\mu-\mu',\alpha\rangle}dt\bigg|^2
=
\sum_{\mu=\mu'}
\bigg|\int_0^L e^{2\pi it\langle\mu-\mu',\alpha\rangle}dt\bigg|^2
\\+
\sum_{\substack{\mu\neq\mu' \\ \mu-\mu'\perp\alpha}}
\bigg|\int_0^L e^{2\pi it\langle\mu-\mu',\alpha\rangle}dt\bigg|^2
+
\sum_{A_\alpha}
\bigg|\int_0^L e^{2\pi it\langle\mu-\mu',\alpha\rangle}dt\bigg|^2.
\end{multline}
The sum for $\mu=\mu'$ contains $N_m$ summands (cf. \cite{rudwig}, Section 5):
\begin{equation}
\label{diag}
\sum_{\mu=\mu'}
\bigg|\int_0^L e^{2\pi it\langle\mu-\mu',\alpha\rangle}dt\bigg|^2
=
\sum_{\mu}
L^2
=
L^2\cdot N_m.
\end{equation}
By Zygmund's trick \cite{zyg}, there can be at most $N_m$ pairs of lattice points satisfying $\mu-\mu'\perp\alpha$, since on a circle there are at most two chords with given length and direction. Thus, the sum for this range contains at most $N_m$ terms:
\begin{equation}
\label{offdiagperp}
\sum_{\substack{\mu\neq\mu' \\ \mu-\mu'\perp\alpha}}
\bigg|\int_0^L e^{2\pi it\langle\mu-\mu',\alpha\rangle}dt\bigg|^2
=
\sum_{\substack{\mu\neq\mu' \\ \mu-\mu'\perp\alpha}}L^2
\leq
L^2\cdot N_m.
\end{equation}
Given a summand
\begin{equation*}
\bigg|\int_0^L e^{2\pi it\langle\mu-\mu',\alpha\rangle}dt\bigg|^2
\end{equation*}
in the range $(\mu,\mu')\in A_\alpha$, we integrate and apply the triangle inequality:
\begin{equation}
\label{lessthanfrac}
\bigg|\int_0^L e^{2\pi it\langle\mu-\mu',\alpha\rangle}dt\bigg|^2
=
\frac{|e^{2\pi i L\langle\mu-\mu',\alpha\rangle}-1|^2}{4\pi^2\langle\mu-\mu',\alpha\rangle^2}
\leq
\frac{1}{\pi^2}\cdot\frac{1}{\langle\mu-\mu',\alpha\rangle^2}.
\end{equation}
Also by the triangle inequality,
\begin{equation}
\label{lessthan1}
\bigg|\int_0^L e^{2\pi it\langle\mu-\mu',\alpha\rangle}dt\bigg|^2
\ll
1.
\end{equation}
Combining \eqref{lessthanfrac} and \eqref{lessthan1},
\begin{equation}
\label{offdiagA}
\sum_{A_\alpha}
\bigg|\int_0^L e^{2\pi it\langle\mu-\mu',\alpha\rangle}dt\bigg|^2
\ll
\sum_{A_\alpha}
\min\bigg(1,\frac{1}{\langle\mu-\mu',\alpha\rangle^2}\bigg).
\end{equation}
The result follows on replacing \eqref{diag}, \eqref{offdiagperp} and \eqref{offdiagA} into \eqref{splitsum2}.
\end{proof}
\begin{proof}[Proof of Proposition \ref{mainprop}]
By Proposition \ref{approxKR}, Lemma \ref{rsq} and Lemma \ref{sumlp}:
\begin{multline*}
\text{Var}(\mathcal{Z})
\ll
m
\cdot
\mathcal{R}_2(m)
\ll
m
\cdot
\frac{1}{N_m^2}
\sum_{(\mu,\mu')\in\mathcal{E}^2}
\bigg|\int_0^L e^{2\pi it\langle\mu-\mu',\alpha\rangle}dt\bigg|^2
\\
\ll
\frac{m}{N_m^2}
\bigg[
N_m
+
\sum_{A_\alpha}
\min\bigg(1,\frac{1}{\langle\mu-\mu',\alpha\rangle^2}\bigg)
\bigg]
=
\frac{m}{N_m}
+
\frac{m}{N_m^2}
\cdot
\sum_{A_\alpha}
\min\bigg(1,\frac{1}{\langle\mu-\mu',\alpha\rangle^2}\bigg).
\end{multline*}
\end{proof}

\section*{Acknowledgements}
This work was carried out as part of the author's PhD thesis at King's College London, under the supervision of Dr. Igor Wigman. The author's PhD is funded by a Graduate Teaching Assistantship, Department of Mathematics. The author wishes to thank Dr. Igor Wigman for his invaluable guidance, remarks and corrections. The author wishes to thank Prof. Ze{\'e}v Rudnick for suggesting this very interesting problem, and for helpful communications.

\bibliographystyle{plain}
\bibliography{bibfile}

\begin{thebibliography}{10}

\bibitem{azawsc}
Jean-Marc Aza{\"{\i}}s and Mario Wschebor.
\newblock {\em Level sets and extrema of random processes and fields}.
\newblock John Wiley \& Sons, Inc., Hoboken, NJ, 2009.

\bibitem{brahpo}
Jean Bourgain and Ze{\'e}v Rudnick.
\newblock On the geometry of the nodal lines of eigenfunctions of the
  two-dimensional torus.
\newblock {\em Ann. Henri Poincar\'e}, 12(6):1027--1053, 2011.

\bibitem{bourud}
Jean Bourgain and Ze{\'e}v Rudnick.
\newblock Nodal intersections and l p restriction theorems on the torus.
\newblock {\em Israel Journal of Mathematics}, 207(1):479--505, 2015.

\bibitem{cilcor}
J.~Cilleruelo and A.~C{\'o}rdoba.
\newblock Trigonometric polynomials and lattice points.
\newblock {\em Proc. Amer. Math. Soc.}, 115(4):899--905, 1992.

\bibitem{cilgr1}
Javier Cilleruelo and Andrew Granville.
\newblock Lattice points on circles, squares in arithmetic progressions and
  sumsets of squares.
\newblock In {\em Additive combinatorics}, volume~43 of {\em CRM Proc. Lecture
  Notes}, pages 241--262. Amer. Math. Soc., Providence, RI, 2007.

\bibitem{cilgr2}
Javier Cilleruelo and Andrew Granville.
\newblock Close lattice points on circles.
\newblock {\em Canad. J. Math.}, 61(6):1214--1238, 2009.

\bibitem{cralea}
Harald Cram{\'e}r and M.~R. Leadbetter.
\newblock {\em Stationary and related stochastic processes. {S}ample function
  properties and their applications}.
\newblock John Wiley \& Sons, Inc., New York-London-Sydney, 1967.

\bibitem{harwri}
G.~H. Hardy and E.~M. Wright.
\newblock {\em An introduction to the theory of numbers}.
\newblock The Clarendon Press, Oxford University Press, New York, fifth
  edition, 1979.

\bibitem{jarnik}
Vojt\v{e}ch Jarn{\'{\i}}k.
\newblock \"{U}ber die {G}itterpunkte auf konvexen {K}urven.
\newblock {\em Math. Z.}, 24(1):500--518, 1926.

\bibitem{krkuwi}
Manjunath Krishnapur, P{\"a}r Kurlberg, and Igor Wigman.
\newblock Nodal length fluctuations for arithmetic random waves.
\newblock {\em Ann. of Math. (2)}, 177(2):699--737, 2013.

\bibitem{landau}
Edmund Landau.
\newblock {\em {\"U}ber die Einteilung der positiven ganzen Zahlen in vier
  Klassen nach der Mindestzahl der zu ihrer additiven Zusammensetzung
  erforderlichen Quadrate}.
\newblock 1909.

\bibitem{rudwig}
Ze{\'e}v Rudnick and Igor Wigman.
\newblock Nodal intersections for random eigenfunctions on the torus.
\newblock {\em arXiv preprint arXiv:1402.3621}, 2014.

\bibitem{ruwiye}
Ze{\'e}v Rudnick, Igor Wigman, and Nadav Yesha.
\newblock Nodal intersections for random waves on the 3-dimensional torus.
\newblock {\em arXiv preprint arXiv:1501.07410}, 2015.

\bibitem{zyg}
A.~Zygmund.
\newblock On {F}ourier coefficients and transforms of functions of two
  variables.
\newblock {\em Studia Math.}, 50:189--201, 1974.

\end{thebibliography}

\end{document}